\documentclass[11pt]{article}

\usepackage{amsmath}
\usepackage{amssymb}
\usepackage{amsthm}
\usepackage{indentfirst}
\usepackage{enumerate}
\usepackage{cite}
\usepackage{color}
\usepackage[T1]{fontenc}
\usepackage{graphicx}
\usepackage{float}
\usepackage[makeroom]{cancel}
\usepackage{bigints}

\newtheorem{dfn}{Definition}[section]
\newtheorem{thm}{Theorem}[section]
\newtheorem{lem}{Lemma}[section]
\newtheorem{cor}{Corollary}[section]
\newtheorem{rem}{Remark}[section]

\numberwithin{equation}{section}

\title{\bf EINSTEIN-HILBERT ACTIONS WITH TORSION}

\author{Nenad O. Vesi\'c\,${}^a$ and Dragoljub D. Dimitrijevi\'c\,${}^b$}

\date{}

\makeatletter
\def\maketag@@@#1{\hbox{\m@th\normalfont\normalsize#1}}
\makeatother

\usepackage{lipsum}

\newcommand\blfootnote[1]{%
  \begingroup
  \renewcommand\thefootnote{}\footnote{#1}%
  \addtocounter{footnote}{-1}%
  \endgroup
}

\begin{document}
  \maketitle

\blfootnote{${}^{a,b}$Faculty of Science and Mathematics, University of Ni\v s}
\blfootnote{${}^a$Department of Mathematics, Serbian Ministry of Education, Science and Technological Development, Grant No. 174012}
\blfootnote{${}^b$Department of Physics, Serbian Ministry of Education, Science and Technological Development, Grant No. 174020, ICTP -- SEENET-MTP NT-03 project "Cosmology-Classical and Quantum Challenges"}

    \begin{abstract}
    In this paper, we studied the full Einstein-Hilbert actions with
    respect to non-symmetric metrics and the corresponding torsion.
    The first concrete result in this paper are the general formulae
    for pressure and density with respect to the Madsen's article
    \big(the equation (3.1), in \cite{madsen1}\big).
    Based on these results, we obtained the expression of
    energy-momentum tensor with respect to non-symmetric metrics.
    We started the generalization of the Bianchi type-I
    model of cosmology with respect to the corresponding
    non-symmetric metrics.\\[3pt]


    \noindent\textbf{$2010$ Math. Subj. Classification:}
    53B50, 58Z05, 46G05
  \end{abstract}

    \section{Introduction}

  Cosmology aims to explain the origin and evolution of the Universe, the underlying physical processes, and to obtain a deeper
understanding of the laws of physics \cite{weinberg_book,
mukhanov_book}. We have only one universe to study, and we cannot
make experiments with it, only observations.

Cosmology is based on the Einstein's theory of general relativity,
i.e. theoretical studies take place in the context of gravitational
theories based on Einstein's theory of general relativity. Spacetime
and hence the evolution of the Universe is determined by the matter
present via the Einstein's equations for gravitational field.

There are three ideas underlying Einstein's theory of general
relativity. The first is that spacetime may be described as a
curved, four-dimensional pseudo-Riemannian manifold. The laws of
physics must be expressed in a form that is valid independently of
any coordinate system used to label points in spacetime. The second
essential idea underlying the general relativity is that at every
spacetime point there exist locally inertial reference frames,
corresponding to locally flat coordinates (carried by freely falling
observers), in which the physics of general relativity is locally
indistinguishable from that of special relativity. This is
Einstein's famous strong equivalence principle and it makes general
relativity an extension of special relativity to a curved spacetime.
The third key idea is that mass curves spacetime in a manner
described by the tensor field equations of Einstein.

Einstein's theory of general relativity defines equations for
gravity. It is a system of non-linear partial differential equations
of up to second order for the components $g_{ij}$ of the spacetime
metric tensor $\hat{g}$ \big(see \cite{bojowald_book}\big). They
determine the structure of spacetime in a covariant and
coordinate-independent way. This statement agrees with the first
idea that the laws of physics must be expressed in a form that is
valid independently of any coordinate system.

Very important fact is that the use of geometry provides important
additional insights by which much information can be gained from
Einstein's equations in a systematic way. Spacetime itself is
equipped with a pseudo-Riemannian structure and encodes gravity
(gravitational field) in a geometrical way. Consequently, geometry
(or more precisely differential geometry) provides means to
understand the structure of the spacetime itself.

\subsection{Motivation}

Note that Einstein's gravity is (appropriately) described by the
pseudo-Riemann geometry which is torsion-free. The spacetime metric
represents the gravitational field. The connections are given by the
Christoffel symbol compatible with the metric structure.

Allowing spacetime to have non-zero torsion, which arise naturally
in generalized gauge theories of gravity and in string theory, we
can analyze cosmological models with a such geometry. We will
consider an action of the generalized Einstein-Hilbert-like form
constructed for non-zero torsion case.

In this paper we investigate cosmological aspects of spacetime with
torsion and discuss modified expressions obtained involving torsion
in general relativity. We will not discuss extended Einstein's
general relativity which includes spin, i.e.
Einstein-Cartan-Kibble-Sciama theory of gravity \cite{twbkibble1,
dwsciama1}. We will discuss a model with (dominant) cosmological
perfect fluid described in the usual way. This perfect cosmological
fluid naturally arises as a consequence of non-zero torsion.

This paper is composed of the introduction and the following
sections:

\begin{enumerate}[-]
  \item In the second section, we will recall the necessary results from differential geometry.
  \item In the third section, we will generalize the energy-momentum tensor. This part will start with the expression of the energy-momentum tensor from the Madsen's article \cite{madsen1}. After that, we will consider $4$-dimensional spacetimes equipped with different non-symmetric metrics.
  \item In the fourth section, we will apply the obtained general results to study the Friedmann spacetime the Bianchi type-I spacetime with torsion. Precisely, we will consider the energy-momentum tensor with respect to special non-symmetric metrics.
\end{enumerate}

  \section{Necessary observations from differential
  geometry}\label{section2physarticle1}

      Different Riemannian and generalized Riemannian
  spaces have been studied by a lot of authors. Some of them are L. P. Eisenhart
  \cite{eisGRN1,eisGRN2}, M. Blau \cite{blau}, M. S. Madsen \cite{madsen1}, V. N. Ponomarev,
  A. O. Barvinsky, Y. N. Obukhov \cite{ponomarev1},
   N. S. Sinyukov \cite{sinjukov}, J. Mike\v
  s and his research team \cite{mik5}, Lj.
  S. Velimirovi\'c, S. M. Min\v ci\'c, M. S. Stankovi\'c
  \cite{infdvelminsta1, infdvelminsta2} and many others. S. M. Min\v
  ci\'c \cite{mincic1, mincic2, mincic4} obtained curvature tensors
  for non-symmetric affine connection spaces (the affine connection
  spaces with torsion). Because the generalized Riemannian spaces
  are special non-symmetric affine connection spaces, these results
  will be useful in this article.

  \subsection{Riemannian and generalized Riemannian spaces}

  Let us present definitions and observations
  necessary for further research in this paper.

  \begin{dfn}\emph{\cite{eisGRN1, eisGRN2}}
    An $N$-dimensional manifold $\mathcal M_N$ equipped with the
    non-symmetric metric
    tensor $\hat g$ with the components $g_{ij}=g_{ij}(x^0,\ldots,x^{N-1})$
    is \emph{the generalized Riemannian space $\mathbb{G}\overset g{\mathbb R}_N$}.
  \end{dfn}

  %ovdes

  Because the tensor $\hat g$ is non-symmetric, the symmetric and anti-symmetric part
  of the components $g_{ij}$ are

    \begin{eqnarray}
    g_{\underline{ij}}=\frac12\big(g_{ij}+g_{ji}\big)&\mbox{and}&
    g_{\underset\vee{ij}}=\frac12\big(g_{ij}-g_{ji}\big).
    \label{eq:gsimantisim}
  \end{eqnarray}

  It evidently holds the equality
  $g_{ij}=g_{\underline{ij}}+g_{\underset\vee{ij}}$.

  We guess that the matrix $\big(g_{\underline{ij}}\big)_{N\times
  N}$ is non-singular, i.e.
  $g=\det\big(g_{\underline{ij}}\big)_{N\times N}\neq 0$. For this
  reason, the contravariant metric is the inverse
  matrix $g^{\underline{ij}}=\big(g_{\underline{ij}}\big)_{N\times
  N}^{-1}$.

  The generalized Christoffel symbols of the first kind of the space
  $\mathbb{G}\overset g{\mathbb R}_N$ are

    \begin{equation}
    \Gamma_{i.jk}=\frac12\big(
    \partial_kg_{j\alpha}-\partial_\alpha g_{jk}+\partial_jg_{\alpha
    k}\big).
    \label{eq:genChristoffelsymbol1st}
  \end{equation}

  The affine connection coefficients of the space $\mathbb{G}\overset g{R}_N$
  are the generalized Christoffel symbols of the second kind

  \begin{equation}
    \Gamma^i_{jk}=g^{\underline{i\alpha}}
    \Gamma_{\alpha.jk}=\frac12g^{\underline{i\alpha}}\big(
    \partial_kg_{j\alpha}-\partial_\alpha g_{jk}+\partial_jg_{\alpha k}\big),
    \label{eq:genChristoffelsymbol}
  \end{equation}

  One may easily check that it holds $\Gamma^i_{jk}\neq\Gamma^i_{kj}$.
  For this reason, the symmetric and anti-symmetric parts of the
  affine connection coefficient $\Gamma^i_{jk}$ are

    \begin{eqnarray}
    \Gamma^i_{\underline{jk}}=
    \frac12\big(\Gamma^i_{jk}+\Gamma^i_{kj}\big)&\mbox{and}&
    \Gamma^i_{\underset\vee{jk}}=
    \frac12\big(\Gamma^i_{jk}-\Gamma^i_{kj}\big).
    \label{eq:Christoffelsymsimantisim}
  \end{eqnarray}

  The differences $\overset gT{}^i_{jk}=
  \Gamma^i_{jk}-\Gamma^i_{kj}$ are called the components of the torsion
  tensor of the space $\mathbb{G}\overset g{\mathbb R}_N$.

  \begin{rem}
  The affine connection $\nabla$ of an affine connection space is
  the bilinear transformation of the set of differentiable vector
  spaces on a manifold $\mathcal M$. With respect to the affine
  connections with or without torsion, the corresponding covariant
  derivatives are defined.

  In the case of the geometrical objects $X^{i_1\ldots
  i_p}_{j_1\ldots j_q}$ and $Y^{i_1\ldots
  i_p}_{j_1\ldots j_q}$, the commutator $[X^{i_1\ldots
  i_p}_{j_1\ldots j_q},Y^{i_1\ldots
  i_p}_{j_1\ldots j_q}]$ vanishes.

    In \emph{\cite{ponomarev1}}, the anti-symmetric part
    $\Gamma^i_{\underset\vee{jk}}$ is called the torsion tensor.
    Generally, if $\nabla$ is the affine connection of the space
    $\mathbb{G}\overset g{\mathbb R}_N$ the torsion tensor is defined as

    \begin{equation*}
      \hat T(\hat X,\hat Y)=\nabla_{\hat Y}\hat X-
      \hat \nabla_{\hat X}\hat Y+[\hat X,\hat Y],
    \end{equation*}

    \noindent for the commutator $[\hat X,\hat Y]=\hat X\hat Y-
    \hat Y\hat X$.

    Coordinately, it has the form
    $\overset gT{}^i_{jk}=\Gamma^i_{jk}-\Gamma^i_{kj}$ and we will use this
    definition for components of torsion.
  \end{rem}

  It is easy to prove that it is satisfied the following equations

  \begin{align}
    &\Gamma^i_{\underline{jk}}=\frac12g^{\underline{i\alpha}}
    \big(\partial_kg_{\underline{j\alpha}}-\partial_\alpha g_{\underline{jk}}+
    \partial_jg_{\underline{\alpha k}}\big),\label{eq:Gammasimgsim}\\
    &\Gamma^i_{\underset\vee{jk}}=\frac12g^{\underline{i\alpha}}
    \big(\partial_kg_{\underset\vee{j\alpha}}-\partial_\alpha
     g_{\underset\vee{jk}}+
    \partial_jg_{\underset\vee{\alpha k}}\big).\label{eq:Gammaantisimgantisim}
  \end{align}

  Moreover, the following expressions also hold

  \begin{equation}
    \Gamma_{i.\underset\vee{jk}}=g_{\underline{i\alpha}}
    \Gamma^\alpha_{\underset\vee{jk}}=\frac12\big(
    \partial_kg_{\underset\vee{ji}}-
    \partial_ig_{\underset\vee{jk}}+
    \partial_jg_{\underset\vee{ik}}\big).
    \tag{\ref{eq:Gammaantisimgantisim}'}\label{eq:Gammaantisimgantisim'}
  \end{equation}

  Similarly as in the case of the symmetric and anti-symmetric part
  of metric tensor, we get
  $\Gamma^i_{jk}=\Gamma^i_{\underline{jk}}+\Gamma^i_{\underset\vee{jk}}$.
  Moreover, with respect to the definition of the
  generalized Riemannian space  $\mathbb{G}\overset g{\mathbb R}_N$ (see \cite{eisGRN2, eisGRN1}),
   it holds

  \begin{eqnarray}
    \Gamma^\alpha_{\underline{i\alpha}}=\frac1{2{g}}
    \partial_i\big|{g}\big|&\mbox{and}&
    \Gamma^\alpha_{\underset\vee{i\alpha}}=0.
    \label{eq:Gammatracesimantisim}
  \end{eqnarray}

  The affine connection space equipped with the affine connection
  which affine connection coefficients are
  $\Gamma^i_{\underline{jk}}$ is the Riemannian space $\overset g{\mathbb R}{}_N$
  and it is called the associated space of the space
  $\mathbb{G}\overset g{\mathbb R}{}_N$.

  With respect to the affine connection of the associated space
  $\overset g{\mathbb R}_N$, one kind of covariant derivative is defined as (see
  \cite{eisGRN1, eisGRN2,sinjukov, mik5})

  \begin{equation}
  \aligned
    \overset g\nabla_kX^{i%_0\ldots i_{p-1}
    }_{j%_0\ldots j_{q-1}
    }&=\partial_k
    X^{i%_0\ldots i_{p-1}
    }_{j%_0\ldots j_{q-1}
    }+\Gamma^i_{\underline{\alpha k}}X^\alpha_j-\Gamma^\alpha_{\underline{jk}}X^i_\alpha,
%    +\sum_{r=0}^{p-1}{\Gamma^{i_r}_{\underline{\alpha k}}
%    X^{i_0\ldots i_{r-1}\alpha i_{r+1}\ldots i_{p-1}}_{j_0\ldots
%    j_{q-1}}}-
%    \sum_{s=0}^{q-1}{\Gamma^\alpha_{\underline{j_sk}}
%    X^{i_1\ldots i_{p-1}}_{j_0\ldots j_{s-1}\alpha j_{s+1}\ldots j_{q-1}}},
  \endaligned\label{eq:covderivativesim}
  \end{equation}

  \noindent for a geometrical object $X^{i%_1\ldots i_p
  }_{j%_1\ldots j_q
  }$ of the type $(1,1)$.
  %$(p,q)$.

  There is one identity of Ricci type with regard to the covariant
  derivative $\overset g\nabla$. The components of the corresponding curvature tensor are

  \begin{equation}
    \overset gR{}^i_{jmn}=\partial_n\Gamma^i_{\underline{jm}}-
    \partial_m\Gamma^i_{\underline{jn}}+
    \Gamma^\alpha_{\underline{jm}}\Gamma^i_{\underline{\alpha n}}-
    \Gamma^\alpha_{\underline{jn}}\Gamma^i_{\underline{\alpha m}}.
    \label{eq:RRN}
  \end{equation}

  The components of the Ricci-curvature tensor and scalar curvature for the space
  $\overset g{\mathbb R}{}_N$ are

  \begin{eqnarray}
    \overset gR{}_{ij}=\overset gR{}^\alpha_{ij\alpha}=
    \partial_\alpha\Gamma^\alpha_{\underline{ij}}-
    \partial_j\Gamma^\alpha_{\underline{i\alpha}}+
    \Gamma^\beta_{\underline{ij}}\Gamma^\alpha_{\underline{\beta\alpha}}-
    \Gamma^\beta_{\underline{i\alpha}}\Gamma^\alpha_{\underline{j\beta}}&
    \mbox{and}&
    \overset gR=\overset gR{}_{\alpha\beta}g^{\underline{\alpha\beta}}.
    \label{eq:RicciScalarRN}
  \end{eqnarray}

  With respect to the affine connection $\overset g{\widetilde\nabla}$ of the
  generalized Riemannian space $\mathbb{G}\overset g{\mathbb R}_N$, four
  kinds of the covariant derivative are defined \cite{mincic1, mincic2,
  mincic4}

  \begin{align}
    &\underset0{\overset g{\widetilde\nabla}}{}_kX^{i%_0\ldots i_{p-1}
    }_{j
    %_0\ldots j_{q-1}
    }=\partial_k
    X^{i%_0\ldots i_{p-1}
    }_{j%_0\ldots j_{q-1}
    }+\Gamma^i_{\alpha k}X^\alpha_j-\Gamma^\alpha_{jk}X^i_\alpha
%    +\sum_{r=0}^{p-1}{\Gamma^{i_r}_{{\alpha k}}
%    X^{i_0\ldots i_{r-1}\alpha i_{r+1}\ldots i_{p-1}}_{j_0\ldots
%    j_{q-1}}}-
%    \sum_{s=0}^{q-1}{\Gamma^\alpha_{{j_sk}}
%    X^{i_0\ldots i_{p-1}}_{j_0\ldots j_{s-1}\alpha j_{s+1}\ldots j_{q-1}}}
,
  \label{eq:covderivativensim1}\\\displaybreak[0]
    &\underset1{\overset g{\widetilde\nabla}}{}_kX^{i%_0\ldots i_{p-1}
    }_{j%_0\ldots j_{q-1}
    }=\partial_k
    X^{%i_0\ldots i_{p-1}
    }_{j%_0\ldots j_{q-1}
    }+\Gamma^i_{k\alpha}X^\alpha_j-
    \Gamma^\alpha_{kj}X^i_\alpha
    %+\sum_{r=1}^p{\Gamma^{i_r}_{{k\alpha}}
    %X^{i_0\ldots i_{r-1}\alpha i_{r+1}\ldots i_{p-1}}_{j_0\ldots
    %j_{q-1}}}-
    %\sum_{s=0}^{q-1}{\Gamma^\alpha_{{kj_s}}
    %X^{i_0\ldots i_{p-1}}_{j_0\ldots j_{s-1}\alpha j_{s+1}\ldots j_{q-1}}}
    ,
  \label{eq:covderivativensim2}\\\displaybreak[0]
  &\underset2{\overset g{\widetilde\nabla}}{}_kX^{i%_0\ldots i_{p-1}
  }_{j%_0\ldots j_{q-1}
  }=
  \partial_k
    X^{i%_0\ldots i_{p-1}
    }_{j%_0\ldots j_{q-1}
    }+\Gamma^i_{\alpha k}X^\alpha_j-\Gamma^\alpha_{kj}X^i_\alpha
%    +\sum_{r=0}^{p-1}{\Gamma^{i_r}_{{\alpha k}}
%    X^{i_0\ldots i_{r-1}\alpha i_{r+1}\ldots i_{p-1}}_{j_0\ldots
%    j_{q-1}}}-
%    \sum_{s=0}^{q-1}{\Gamma^\alpha_{{kj_s}}
%    X^{i_0\ldots i_{p-1}}_{j_0\ldots j_{s-1}\alpha j_{s+1}\ldots j_{q-1}}}
,
  \label{eq:covderivativensim3}\\\displaybreak[0]
  &\underset3{\overset g{\widetilde\nabla}}{}_kX^{i%_0\ldots i_{p-1}
  }_{j%_0\ldots j_{q-1}
  }
  =\partial_k
    X^{i%_0\ldots i_{p-1}
    }_{j%_0\ldots j_{q-1}
    }
    +\Gamma^i_{k\alpha}X^\alpha_j-\Gamma^\alpha_{jk}X^i_\alpha
%    +\sum_{r=0}^{p-1}{\Gamma^{i_r}_{{k\alpha}}
%    X^{i_0\ldots i_{r-1}\alpha i_{r+1}\ldots i_{p-1}}_{j_0\ldots
%    j_{q-1}}}-
%    \sum_{s=0}^{q-1}{\Gamma^\alpha_{{j_sk}}
%    X^{i_0\ldots i_{p-1}}_{j_0\ldots j_{s-1}\alpha j_{s+1}\ldots j_{q-1}}}
.
  \label{eq:covderivativensim4}
  \end{align}

  Based on these covariant derivatives, four
  curvature tensors, eight derived curvature tensors and fifteen
  curvature pseudotensors are obtained \cite{mincic1, mincic2, mincic4}.

  In this paper, we will deal with the lagrangian
  obtained with respect to the curvature and derived curvature tensors. The components of the
  curvature tensors and the derived curvature tensors for the space $\mathbb G\overset g{\mathbb R}{}_N$ are elements of
  the family
  \cite{z4}

  \begin{equation}
    \overset g{\widetilde R}{}^i_{jmn}=
    \overset gR{}^i_{jmn}+u\overset g\nabla{}_n\overset gT{}^i_{{jm}}
    +u'\overset g\nabla{}_m\overset gT{}^i_{{jn}}+
    v\overset gT{}^\alpha_{{jm}}\overset gT{}^i_{{\alpha
    n}}+
    v'\overset gT{}^\alpha_{{jn}}\overset gT{}^i_{{\alpha
    m}}+
    w\overset gT{}^\alpha_{{mn}}\overset gT{}^i_{{\alpha
    j}},
    \label{eq:hatRGRN(T)}
  \end{equation}

  \noindent for real coefficients $u,u',v,v',w$ and the components $\overset gT{}^i_{jk}$ of the torsion tensor $\hat{\overset gT}$.

  With respect to the equations (\ref{eq:Gammatracesimantisim},
  \ref{eq:hatRGRN(T)}), one obtains that the families of Ricci-curvature tensors and the
  scalar curvature of the space $\mathbb{G}\overset g{\mathbb R}_N$ are

  \begin{eqnarray}
    \overset g{\widetilde R}{}_{ij}=\overset gR{}_{ij}+
    u\overset g\nabla{}_\alpha\overset g T{}^\alpha_{ij}-
    (v'+w)\overset gT{}^\alpha_{i\beta}\overset gT{}^\beta_{j\alpha}&\mbox{and}&
    \overset g{\widetilde
    R}=\overset gR-(v'+w)g^{\underline{\gamma\delta}}
    \overset gT{}^\alpha_{\gamma\beta}
    \overset gT{}^\beta_{\delta\alpha}.
    \label{eq:RicciScalarGRN}
  \end{eqnarray}

  Six of the curvature tensors in the family (\ref{eq:hatRGRN(T)})
  are linearly independent, for example

  \begin{align}
    &\underset0{\overset g{\widetilde R}}{}^i_{jmn}=\overset gR{}^i_{jmn}+
    \frac12\overset g\nabla{}_n\overset gT{}^i_{jm}-
    \frac12\overset g\nabla{}_m\overset gT{}^i_{jn}+
    \frac14\overset gT{}^\alpha_{jm}\overset gT{}^i_{\alpha n}-
    \frac14\overset gT{}^\alpha_{jn}\overset gT{}^i_{\alpha
    m},\label{eq:R1phys}\\\displaybreak[0]
    &\underset1{\overset g{\widetilde R}}{}^i_{jmn}=\overset gR{}^i_{jmn}-
    \frac12\overset g\nabla{}_n\overset gT{}^i_{jm}+\frac12
    \overset g\nabla{}_m\overset gT{}^i_{jn}+
    \frac14\overset gT{}^\alpha_{jm}\overset gT{}^i_{\alpha n}-
    \frac14\overset gT{}^\alpha_{jn}\overset gT{}^i_{\alpha
    m},\label{eq:R2phys}\\\displaybreak[0]
    &\underset2{\overset g{\widetilde R}}{}^i_{jmn}=\overset gR{}^i_{jmn}+
    \frac12\overset g\nabla{}_n\overset gT{}^i_{jm}+
    \frac12\overset g\nabla{}_m\overset gT{}^i_{jn}-
    \frac14\overset gT{}^\alpha_{jm}\overset gT{}^i_{\alpha n}+
    \frac14\overset gT{}^\alpha_{jn}\overset gT{}^i_{\alpha
    m}-\frac12\overset gT{}^\alpha_{mn}\overset gT{}^i_{\alpha j},\label{eq:R3phys}\\\displaybreak[0]
    &\underset3{\overset g{\widetilde R}}{}^i_{jmn}=\overset gR{}^i_{jmn}+
    \frac12\overset g\nabla{}_n\overset gT{}^i_{jm}+\frac12
    \overset g\nabla{}_m\overset gT{}^i_{jn}-
    \frac14\overset gT{}^\alpha_{jm}\overset gT{}^i_{\alpha n}+
    \frac14\overset gT{}^\alpha_{jn}\overset gT{}^i_{\alpha
    m}+\frac12\overset gT{}^\alpha_{mn}\overset gT{}^i_{\alpha j},\label{eq:R4phys}\\\displaybreak[0]
    &\underset4{\overset g{\widetilde R}}{}^i_{jmn}=\overset gR{}^i_{jmn}-
    \frac14\overset gT{}^\alpha_{jm}\overset gT{}^i_{\alpha n}+
    \frac14\overset gT{}^\alpha_{jn}\overset gT{}^i_{\alpha
    m},\label{eq:R5phys}\\\displaybreak[0]
    &\underset5{\overset g{\widetilde R}}{}^i_{jmn}=\overset gR{}^i_{jmn}+
    \frac14\overset gT{}^\alpha_{jm}\overset gT{}^i_{\alpha n}+
    \frac14\overset gT{}^\alpha_{jn}\overset gT{}^i_{\alpha
    m}.\label{eq:R6phys}
  \end{align}

  The components of the corresponding Ricci-curvature tensors of the space
  $\mathbb{G}\overset g{\mathbb R}_N$ are

  \begin{equation}
    \begin{array}{ll}
      \underset0{\overset g{\widetilde R}}{}_{ij}=\overset gR{}_{ij}+
      \frac12\overset g\nabla{}_\alpha \overset gT{}^\alpha_{ij}+\frac14
      \overset gT{}^\beta_{i\alpha}\overset gT{}^\alpha_{j\beta}&
      \underset1{\overset g{\widetilde R}}{}_{ij}=\overset gR{}_{ij}-
      \frac12\overset g\nabla{}_\alpha\overset gT{}^\alpha_{ij}+
      \frac14\overset gT{}^\beta_{i\alpha}\overset gT{}^\alpha_{j\beta},\\
      \underset2{\overset g{\widetilde R}}{}_{ij}=\overset gR{}_{ij}+
      \frac12\overset g\nabla{}_\alpha
      \overset gT{}^\alpha_{ij}+\frac14\overset gT{}^\beta_{i\alpha}
      \overset gT{}^\alpha_{j\beta},&
      \underset3{\overset g{\widetilde R}}{}_{ij}=\overset gR{}_{ij}+
      \frac12\overset g\nabla{}_\alpha
      \overset gT{}^\alpha_{ij}-\frac34\overset gT{}^\beta_{i\alpha}
      \overset gT{}^\alpha_{j\beta},\\
      \underset4{\overset g{\widetilde
      R}}{}_{ij}=\overset gR{}_{ij}-\frac14
      \overset gT{}^\beta_{i\alpha}\overset gT{}^\alpha_{j\beta},&
      \underset5{\overset g{\widetilde R}}{}_{ij}=\overset gR{}_{ij}.
    \end{array}\label{eq:RiccilinindependentGRN}
  \end{equation}

  Three of them, for instance $\underset0{\overset g{\widetilde R}}{}_{ij},
  \underset1{\overset g{\widetilde R}}{}_{ij},
  \underset 5{\overset g{\widetilde R}}{}_{ij}$, are linearly
  independent.

  The corresponding scalar curvatures of the space $\mathbb{G}\overset g{\mathbb R}{}_N$
  are

    \begin{equation}
    \begin{array}{ll}
      \underset0{\overset g{\widetilde R}}=\overset gR+
      \frac14g^{\underline{\gamma\delta}}
      \overset gT{}^\beta_{\gamma\alpha}\overset gT{}^\alpha_{\delta\beta}&
      \underset1{\overset g{\widetilde R}}=\overset gR+
      \frac14g^{\underline{\gamma\delta}}\overset gT{}^\beta_{\gamma\alpha}\overset gT{}^\alpha_{\delta\beta},\\
      \underset2{\overset g{\widetilde R}}=\overset gR
      +\frac14g^{\underline{\gamma\delta}}
      \overset gT{}^\beta_{\gamma\alpha}\overset gT{}^\alpha_{\delta\beta},&
      \underset3{\overset g{\widetilde R}}=\overset gR-
      \frac34g^{\underline{\gamma\delta}}
      \overset gT{}^\beta_{\gamma\alpha}\overset gT{}^\alpha_{\delta\beta},\\
      \underset4{\overset g{\widetilde
      R}}=\overset gR-\frac14g^{\underline{\gamma\delta}}
      \overset gT{}^\beta_{\gamma\alpha}\overset gT{}^\alpha_{\delta\beta},&
      \underset5{\overset g{\widetilde R}}=\overset gR.
    \end{array}\label{eq:ScalarlinindependentGRN}
  \end{equation}

  Two of the scalar curvatures, e.g. $\underset0{\overset g{\widetilde
  R}},\underset 5{\overset g{\widetilde R}}=\overset gR$, are linearly independent.

  In our research about physics, we will stay focused on the
  generalized Riemannian space $\mathbb{G}\overset g{\mathbb R}_4$ equipped with
   the metric
  tensor

  \begin{equation}
    \hat g=\left[\begin{array}{rrrc}
      s_0(t)&n_0(t)&n_1(t)&n_2(t)\\
      -n_0(t)&s_1(t)&n_3(t)&n_4(t)\\
      -n_1(t)&-n_3(t)&s_2(t)&n_5(t)\\
      -n_2(t)&-n_4(t)&-n_5(t)&s_3(t)
    \end{array}\right],\label{eq:metricsPhys4}
  \end{equation}

  \noindent for differentiable functions $s_1(t),s_2(t), s_3(t),
  n_0(t),\ldots,n_5(t)$, depending on the time coordinate $t$.  The variables in
  the space $\mathbb{GR}_4$ will be $x^0=t$ and the variables
  $x^1,x^2,x^3$ are the space variables.

  \subsection{Variations}

  We will recall the infinitesimal deformations as in \cite{infdvelminsta1,
  infdvelminsta2} in this part of the paper. A
  necessary rule will be obtained in here.

  A transformation $f:\mathbb{G}\overset g{\mathbb R}_N\to\mathbb{G}\overset g{\overline R}_N$
  defined as $x\equiv(x^i)\to\overline x\equiv(\overline x^i)$, for

  \begin{equation}
    \overline x^i=x^i+\varepsilon z^i(x^j),\quad i,j=0,\ldots,N-1,
    \label{eq:infdeformationx}
  \end{equation}

  \noindent where $\varepsilon$ is an infinitesimal, is the
  infinitesimal deformation of the space $\mathbb{G}\overset g{\mathbb R}_N$ determined
  by the vector field $z=z(x^i)$.

  A local coordinate system in which the point $x$ is endowed with
  coordinates $x^i$ and the point $\overline x$ with
  the coordinates $\overline x^i$ will be denoted by $(s)$. In
  another coordinate system $(s')$, corresponding to the point
  $x=(x^i)$ new coordinates $x^{i'}=\overline x^i$, i.e. as new
  coordinates $x^{i'}$ of the point $x=(x^i)$ we choose old
  coordinates of the point $\overline x=(\overline x^i)$. In other
  words, the equalities $x=(x^{i'})=(\overline x^i)$ are satisfied
  at the system $(s')$.

  A considered geometric object $\mathcal A$ with respect to the
  system $(s)$ at the point $x=x^i$ will be denoted as $\mathcal
  A(s,x)$.

  The point $\overline x$ is said to be deformed point of the point
  $x$, if the equation (\ref{eq:infdeformationx}) holds. Geometrical
  object $\overline{\mathcal A}(s,x)$ is the deformed object
  $\mathcal A(s,x)$ with respect to the deformation
  (\ref{eq:infdeformationx}), if it holds

  \begin{equation}
    \overline{\mathcal A}(s',x)=\mathcal A(s,\overline x),
    \label{eq:infdeformationsA}
  \end{equation}

  \noindent for the coordinate systems $s$ and $s'$.

  The limit

  \begin{equation}
    \delta\mathcal A=\lim_{\varepsilon\to0}{\frac{\overline{\mathcal A}(s,x)
    -\mathcal
    A(s,x)}\varepsilon}
    \label{eq:variation}
  \end{equation}

  \noindent is the (first) variation of the geometric object
  $\mathcal A(s,x)$. We may notice $\overline{\mathcal A}(s,x)=
  \mathcal A(s,x)+\varepsilon\delta\mathcal A$.

    Let

  \begin{equation}
    \mathcal I[f]=\int{F\big[t,f(t),f'(t)\big]dt},
    \label{eq:functional}
  \end{equation}

  \noindent be a functional for $f$, where $f'(t)=df/dt$. If $f$ is varied by
  adding to it a function $\delta f$, and the integrand
  $F(x,f+\delta f,f'+\delta f')$ is expanded in powers of $\delta
  f$, then the change in the value of $\mathcal I$ to first order in $\delta
  f$ is

  \begin{equation}
    \delta\mathcal I[f]=\int{\left(
    \frac{\partial F}{\partial f}-
    \frac d{dt}\frac{\partial F}{\partial f'}\right)\delta qdt}.
    \label{eq:eulerlagrangeequation}
  \end{equation}

  The function $\delta F/\delta f(t)$ is the functional derivative
  of $F$ with respect to $f$ at the point $t$. This functional
  derivative may be computed as

  \begin{equation}
    \frac{\delta F}{\delta f(t)}=\frac{\partial F}{\partial f}-
    \frac d{dt}\frac{\partial F}{\partial f'}.
    \label{eq:eulerlagrangeequation}
  \end{equation}

  As we may conclude, if values of a function $F$ do not change when
  we change values of a function $f(t)$, with respect to $t$, then
  the functional derivative of $F$ with respect to $f$ vanishes. In
  other words, if the function $F$ is not expressed as a composition
  of the function $f(t)$ and some other function $g$, then the
  functional derivative $\delta F/\delta f(x)$ is equal $0$.

  With respect to the equation (\ref{eq:eulerlagrangeequation}), we may define
  the functional derivative of a tensor $\hat X$ with components $X^{i_0\ldots
  i_p}_{j_0\ldots j_q}$ with respect to a tensor $\hat Y$ with components $Y^{m_0\ldots
  m_u}_{n_0\ldots n_v}$.

  First of all, notice that the component $X^{i_0\ldots
  i_p}_{j_0\ldots j_q}$ of the tensor $\hat X$ is expressed as

  \begin{equation}
    X^{i_0\ldots
  i_p}_{j_0\ldots j_q}=\sum_{u=0}^p{d_{(u)}\cdot X^{(u)}}+
  \sum_{v=0}^q{\partial_{(v)}X_{(v)}},
  \label{eq:tensorexpression}
  \end{equation}

  \noindent for the orthogonal base $(d_u)\equiv(d/dx{}_u),u=0,\ldots,p$, of the space of functionals and the
  base $(\partial_v)=(\partial/\partial x^v),v=0,\ldots,q$, of the space of
  positions.

  If we treat the components $X^{i_0\ldots i_p}_{j_0\ldots j_q}$ and
  $Y^{m_0\ldots m_u}_{n_0\ldots n_v}$ as functions
  $X^{i_0\ldots i_p}_{j_0\ldots j_q}(\varepsilon)$ and
  $Y^{m_0\ldots m_u}_{n_0\ldots n_v}(\varepsilon)$ of  an
  infinitesimal deformation, we may conclude that the function $F$
  from the equation (\ref{eq:functional}) is

  \begin{equation}
    F\left[\varepsilon,Y^{m_0\ldots m_u}_{n_0\ldots n_v}(\varepsilon),
    d\big(Y^{m_0\ldots m_u}_{n_0\ldots
    n_v}(\varepsilon)\big)/d\varepsilon\right]=
    F\left[\varepsilon,Y^{m_0\ldots m_u}_{n_0\ldots n_v}(\varepsilon),
    \delta Y^{m_0\ldots m_u}_{n_0\ldots
    n_v}\right].
    \label{eq:tensorfunctional}
  \end{equation}

  After comparing the equations (\ref{eq:eulerlagrangeequation}),
  (\ref{eq:tensorfunctional}), we will conclude that if the
  components $X^{i_0\ldots i_p}_{j_0\ldots j_q}(\varepsilon)$ and $Y^{m_0\ldots
  m_u}_{n_0\ldots n_v}(\varepsilon)$ as well as the components
  $X^{i_0\ldots i_p}_{j_0\ldots j_q}(\varepsilon)$ and $\delta Y^{m_0\ldots
  m_u}_{n_0\ldots n_v}$ are functionally independent, the variational
  derivative of the component $X^{i_0\ldots i_p}_{j_0\ldots j_q}$
  with respect to the component $Y^{m_0\ldots m_u}_{n_0\ldots n_v}$
  is equal zero. Moreover, if the tensor $\hat X$ is not expressed
  as a composition of the tensor $\hat Y$ and arbitrary tensor $\hat Z$, the
  functional derivative of the tensor $\hat X$ with respect to the
  tensor $\hat Y$ is equal zero.

    \section{Theoretical considerations}\label{section3physarticle1}

    In this section, we will theoretically consider the
    Einstein-Hilbert action with torsion
    and the corresponding energy-momentum tensor.

  \subsection{Four-dimensional spacetime}

  Let $\mathbb{G}\overset g{\mathbb R}_4$ be the generalized Riemannian space equipped
  with the metric tensor (\ref{eq:metricsPhys4}). The symmetric and
  anti-symmetric parts of this tensor are

  \begin{eqnarray}
    \underline{\hat g}=\left[\begin{array}{cccc}
      s_0(t)&0&0&0\\
      0&s_1(t)&0&0\\
      0&0&s_2(t)&0\\
      0&0&0&s_3(t)
    \end{array}\right],&
    \underset{{\vee}}{\hat g}=\left[\begin{array}{cccc}
      0&n_0(t)&n_1(t)&n_2(t)\\
      -n_0(t)&0&n_3(t)&n_4(t)\\
      -n_1(t)&-n_3(t)&0&n_5(t)\\
      -n_2(t)&-n_4(t)&-n_5(t)&0
    \end{array}\right].
    \label{eq:metricsPhys4simantisim}
  \end{eqnarray}

  The contravariant metric tensor $\hat{\underline g}{}^{-1}$ \big(with the components
   $g^{\underline{ij}}$\big) is

  \begin{equation}
    \underline{\hat g}{}^{-1}=\left[\begin{array}{cccc}
      \big(s_0(t)\big)^{-1}&0&0&0\\
      0&\big(s_1(t)\big)^{-1}&0&0\\
      0&0&\big(s_2(t)\big)^{-1}&0\\
      0&0&0&\big(s_3(t)\big)^{-1}
    \end{array}\right].
    \label{eq:metricsPhys4simcontravariant}
  \end{equation}

  The generalized Christoffel symbols of the
  first and the second kind of the space $\mathbb{G}\overset g{\mathbb R}_4$
  are

  \begin{eqnarray}
    \Gamma_{i.jk}=\frac12\Big(
    \partial_kg_{{ji}}-
    \partial_ig_{{jk}}+
    \partial_jg_{{ik}}\Big)&\mbox{and}&
    \Gamma^i_{jk}=\frac12\big(s_i(t)\big)^{-1}\Big(
    \partial_kg_{{ji}}-
    \partial_ig_{{jk}}+
    \partial_jg_{{ik}}\Big).
    \label{eq:GammaPhys4}
  \end{eqnarray}

The covariant anti-symmetric Christoffel symbols are

  \begin{scriptsize}
  \begin{eqnarray}
    \begin{array}{llllll}
      \Gamma_{0.\underset\vee{12}}=-\frac12n_3'(t),&
      \Gamma_{0.\underset\vee{21}}=\frac12n_3'(t),&
      \Gamma_{1.\underset\vee{02}}=\frac12n_3'(t),&
      \Gamma_{1.\underset\vee{20}}=-\frac12n_3'(t),&
      \Gamma_{2.\underset\vee{01}}=-\frac12n_3'(t),&
      \Gamma_{2.\underset\vee{10}}=\frac12n_3'(t),\\
      \Gamma_{0.\underset\vee{13}}=-\frac12n_4'(t),&
      \Gamma_{0.\underset\vee{31}}=\frac12n_4'(t),&
      \Gamma_{1.\underset\vee{03}}=\frac12n_4'(t),&
      \Gamma_{1.\underset\vee{30}}=-\frac12n_4'(t),&
      \Gamma_{3.\underset\vee{01}}=-\frac12n_4'(t),&
      \Gamma_{3.\underset\vee{10}}=\frac12n_4'(t),\\
      \Gamma_{0.\underset\vee{23}}=-\frac12n_5'(t),&
      \Gamma_{0.\underset\vee{32}}=\frac12n_5'(t),&
      \Gamma_{2.\underset\vee{03}}=\frac12n_5'(t),&
      \Gamma_{2.\underset\vee{30}}=-\frac12n_5'(t),&
      \Gamma_{3.\underset\vee{02}}=-\frac12n_5'(t),&
      \Gamma_{3.\underset\vee{20}}=\frac12n_5'(t),
    \end{array}\label{eq:Gammacovariantantisim}
  \end{eqnarray}
  \end{scriptsize}

  \noindent and $\Gamma_{i.\underset\vee{jk}}=0$ in all other cases.

    The corresponding anti-symmetric parts
    of the generalized Christoffel symbols $\Gamma^i_{jk}$ are

%  \begin{equation}
%    \begin{array}{ll}
\begin{align*}
      &\Gamma^0_{\underset\vee{12}}=g^{\underline{0\alpha}}
      \Gamma_{\alpha.\underset\vee{12}}=
      -\frac12{n_3'(t)}\big(s_0(t)\big)^{-1},&
      \Gamma^0_{\underset\vee{21}}=\frac12{n_3'(t)}\big(s_0(t)\big)^{-1},\\\displaybreak[0]
      &\Gamma^0_{\underset\vee{13}}=g^{\underline{0\alpha}}
      \Gamma_{\alpha.\underset\vee{13}}=-\frac12{n_4'(t)}\big(s_0(t)\big)^{-1},&
      \Gamma^0_{\underset\vee{31}}=\frac12{n_4'(t)}\big(s_0(t)\big)^{-1},\\\displaybreak[0]
      &\Gamma^0_{\underset\vee{23}}=g^{\underline{0\alpha}}
      \Gamma_{\alpha.\underset\vee{23}}=-\frac12{n_5'(t)}\big(s_0(t)\big)^{-1},&
      \Gamma^0_{\underset\vee{32}}=\frac12{n_5'(t)}\big(s_0(t)\big)^{-1},\\\displaybreak[0]
      &\Gamma^1_{\underset\vee{20}}=g^{\underline{1\alpha}}
      \Gamma_{\alpha.\underset\vee{20}}=-\frac12{n_3'(t)}{\big(s_1(t)\big)^{-1}},&
      \Gamma^1_{\underset\vee{02}}=\frac12{n_3'(t)}{\big(s_1(t)\big)^{-1}},\\\displaybreak[0]
      &\Gamma^1_{\underset\vee{30}}=g^{\underline{1\alpha}}
      \Gamma_{\alpha.\underset\vee{30}}=-\frac12{n_4'(t)}{\big(s_1(t)\big)^{-1}},&
      \Gamma^1_{\underset\vee{03}}=\frac12{n_4'(t)}{\big(s_1(t)\big)^{-1}},\\\displaybreak[0]
      &\Gamma^2_{\underset\vee{01}}=g^{\underline{2\alpha}}
      \Gamma_{\alpha.\underset\vee{01}}=-\frac12{n_3'(t)}{\big(s_2(t)\big)^{-1}},&
      \Gamma^2_{\underset\vee{10}}=\frac12{n_3'(t)}{\big(s_2(t)\big)^{-1}},\\\displaybreak[0]
      &\Gamma^2_{\underset\vee{30}}=g^{\underline{2\alpha}}
      \Gamma_{\alpha.\underset\vee{30}}=-\frac12{n_5'(t)}{\big(s_2(t)\big)^{-1}},&
      \Gamma^2_{\underset\vee{03}}=\frac12{n_5'(t)}{\big(s_2(t)\big)^{-1}},\\\displaybreak[0]
      &\Gamma^3_{\underset\vee{01}}=g^{\underline{3\alpha}}
      \Gamma_{\alpha.\underset\vee{01}}=-\frac12{n_4'(t)}{\big(s_3(t)\big)^{-1}},&
      \Gamma^3_{\underset\vee{10}}=\frac12{n_4'(t)}{\big(s_3(t)\big)^{-1}},\\\displaybreak[0]
      &\Gamma^3_{\underset\vee{02}}=g^{\underline{3\alpha}}
      \Gamma_{\alpha.\underset\vee{02}}=-\frac12{n_5'(t)}{\big(s_3(t)\big)^{-1}},&
      \Gamma^3_{\underset\vee{20}}=\frac12{n_5'(t)}{\big(s_3(t)\big)^{-1}},
\end{align*}
%    \end{array}
%  \end{equation}

  \noindent and $\Gamma^i_{\underset\vee{jk}}=0$ in all other cases.

    Because $\Gamma^\gamma_{\underset\vee{i\delta}}
  \Gamma^\delta_{\underset\vee{j\gamma}}=\Gamma^\gamma_{\underset\vee{j\delta}}
  \Gamma^\delta_{\underset\vee{i\gamma}}$, it is
  enough to obtain the following geometrical objects for further
  calculations:

\begin{scriptsize}
  \begin{align}
    &\aligned\Gamma^\gamma_{\underset\vee{0\delta}}
  \Gamma^\delta_{\underset\vee{0\gamma}}&=
  -\frac12\left({\big(n_3'(t)\big)^2}{\big(s_1(t)\big)^{-1}
  \big(s_2(t)\big)^{-1}}+
  {\big(n_4'(t)\big)^2}{\big(s_1(t)\big)^{-1}\big(s_3(t)\big)^{-1}}+
  {\big(n_5'(t)\big)^2}{\big(s_2(t)\big)^{-1}\big(s_3(t)\big)^{-1}}\right),
  \endaligned\label{eq:genGammaGamma11}\\\displaybreak[0]
  &\aligned\Gamma^\gamma_{\underset\vee{1\delta}}
  \Gamma^\delta_{\underset\vee{1\gamma}}&=
  -\frac12\left({\big(n_3'(t)\big)^2}{\big(s_0(t)\big)^{-1}\big(s_2(t)\big)^{-1}}+
  {\big(n_4'(t)\big)^2}{\big(s_0(t)\big)^{-1}\big(s_3(t)\big)^{-1}}\right),
  \endaligned\label{eq:genGammaGamma22}\\\displaybreak[0]
  &\Gamma^\gamma_{\underset\vee{1\delta}}
  \Gamma^\delta_{\underset\vee{2\gamma}}=
  -\frac12{n_4'(t)n_5'(t)}{\big(s_0(t)\big)^{-1}\big(s_3(t)\big)^{-1}},\label{eq:genGammaGamma23}\\\displaybreak[0]
  &\Gamma^\gamma_{\underset\vee{1\delta}}
  \Gamma^\delta_{\underset\vee{3\gamma}}=
  \frac12{n_3'(t)n_5'(t)}{\big(s_0(t)\big)^{-1}\big(s_2(t)\big)^{-1}},\label{eq:genGammaGamma24}\\\displaybreak[0]
  &\aligned\Gamma^\gamma_{\underset\vee{2\delta}}
  \Gamma^\delta_{\underset\vee{2\gamma}}&=
  -\frac12\left({\big(n_3'(t)\big)^2}{\big(s_0(t)\big)^{-1}\big(s_1(t)\big)^{-1}}+
  {\big(n_5'(t)\big)^2}{\big(s_0(t)\big)^{-1}\big(s_3(t)\big)^{-1}}\right),
  \endaligned\label{eq:genGammaGamma33}\\\displaybreak[0]
  &\Gamma^\gamma_{\underset\vee{2\delta}}
  \Gamma^\delta_{\underset\vee{3\gamma}}=-
  \frac12{n_3'(t)n_4'(t)}{\big(s_0(t)\big)^{-1}\big(s_1(t)\big)^{-1}},\label{eq:genGammaGamma34}\\\displaybreak[0]
  &\aligned\Gamma^\gamma_{\underset\vee{3\delta}}
  \Gamma^\delta_{\underset\vee{3\gamma}}&=
  -\frac12\left({\big(n_4'(t)\big)^2}{\big(s_0(t)\big)^{-1}\big(s_1(t)\big)^{-1}}+
  {\big(n_5'(t)\big)^2}{\big(s_0(t)\big)^{-1}\big(s_2(t)\big)^{-1}}\right),
  \endaligned\label{eq:genGammaGamma44}
  \end{align}
  \end{scriptsize}

    \noindent and $\Gamma^\gamma_{\underset\vee{i\delta}}
  \Gamma^\delta_{\underset\vee{j\gamma}}=0$ in all other cases for $i\leq j$.

    With respect to the tensor $\hat{\underline g}{}^{-1}$, the above
    mentioned linearly
    independent scalars $\underset0{\overset g{\widetilde R}},\overset gR$
    and the metric determinant $g=s_0(t)s_1(t)s_2(t)s_3(t)$, we obtain that the
  scalar curvature of the space $\mathbb{G}\overset g{R}{}_4$ with torsion is
  \big(see the equations (\ref{eq:ScalarlinindependentGRN})
  and (\ref{eq:metricsPhys4})\big)

  \begin{equation}
    \aligned
    \overset g{\widetilde R}&=
    \overset gR-\frac{3}{2}\cdot g^{-1}\left(s_3(t)\big(n_3'(t)\big)^2+
    s_2(t)\big(n_4'(t)\big)^2+s_1(t)\big(n_5'(t)\big)^2\right).
    \endaligned\label{eq:metriclagrangianwithtorsion}
  \end{equation}

  The scalar curvature (\ref{eq:metriclagrangianwithtorsion}) can be used to define the full Einstein-Hilbert action.

%  \begin{equation}
%    S=\int{d^4x\sqrt{|g|}\left(\frac1{2\kappa}\overset gR+\mathcal
%    L_M\right)}.
%    \label{eq:eheffectgeneral}
%  \end{equation}

  Note that the Lagrangian density in \cite{blau} is $\sqrt{|g|}\overset gR$. Here, we examine the Lagrangian density
   in the case of the space with torsion to be
  $\sqrt{|g|}\underset0{\overset g{\widetilde R}}$.

  We are aimed to express a term
  $\mathcal L_M$ describing a dominant cosmological fluid  appearing in the model
   as the function of torsion.
  The full Lagrangian density of the model takes the form

  \begin{equation}
    \mathcal L=\sqrt{|g|}\left(\frac1{2\kappa}\overset gR+
    \frac14g^{\underline{\gamma\delta}}
    \overset gT{}^\beta_{\gamma\alpha}\overset gT{}^\alpha_{\delta\beta}\right).
    \label{eq:lagrangian}
  \end{equation}

  \begin{rem}
    The same effect would be achieved if we get the metric

    \begin{equation*}
      \hat h=\left[\begin{array}{cccc}
        s_0(t)&2\kappa n_0(t)&2\kappa n_1(t)&
        2\kappa n_2(t)\\
        -2\kappa n_0(t)&s_1(t)&2\kappa n_3(t)&2\kappa n_4(t)\\
        -2\kappa n_1(t)&-2\kappa n_3(t)&s_2(t)&2\kappa n_5(t)\\
        -2\kappa n_2(t)&-2\kappa n_4(t)&-2\kappa n_5(t)&s_3(t)
      \end{array}\right].
    \end{equation*}

    \noindent and the Lagrangian
    $\mathcal L=R+\frac14g^{\underline{\gamma\delta}}\overset gT{}^\alpha_{\gamma\beta}
    \overset gT{}^\beta_{\delta\gamma}$, for $\mathcal L_M=
    \frac14g^{\underline{\gamma\delta}}\overset gT{}^\alpha_{\gamma\beta}
    \overset gT{}^\beta_{\delta\gamma}$.

  Precisely, we guess that
  the symmetric and anti-symmetric part of the metric $g$
  will satisfy the equality $g_{ij}=g_{\underline{ij}}+2\kappa
  g_{\underset\vee{ij}}$ in our paper. It is done for the
  simplifying of the expression through the computation process.
  \end{rem}

  \begin{rem}
    We could easily put $2\kappa=1$. In this way, some researchers
    may find wrong conclusions. To avoid that, we will analyze the
    general case with the constant $\kappa$.
  \end{rem}

      \begin{lem}
      In the space $\mathbb{GR}_4$ equipped with the metrics $$\hat g=
      \left[\begin{array}{cccc}
        g_{\underline{11}}&g_{\underset\vee{12}}&g_{\underset\vee{13}}&g_{\underset\vee{14}}\\
        -g_{\underset\vee{12}}&g_{\underline{22}}&g_{\underset\vee{23}}&g_{\underset\vee{24}}\\
        -g_{\underset\vee{13}}&-g_{\underset\vee{23}}&g_{\underline{33}}&g_{\underset\vee{34}}\\
        -g_{\underset\vee{14}}&-g_{\underset\vee{24}}&-g_{\underset\vee{34}}&g_{\underline{44}}
      \end{array}\right],$$
      the tensor $\overset g{\hat T}$
      whose components are $\overset gT{}_{i.jk}=\frac12\big(\partial_kg_{\underset\vee{ji}}-
      \partial_ig_{\underset\vee{jk}}+
      \partial_jg_{\underset\vee{ik}}\big)$  is not a function of
      neither
      $g_{\underline{pq}}$ nor $\partial_rg_{\underline{pq}}$.
    \end{lem}

    \begin{proof}
    The following equalities are satisfied:

    \begin{eqnarray*}
      \overset gT{}_{i.jk}=-\overset gT{}_{i.kj},&
      \overset gT{}_{i.jk}=-\overset gT{}_{j.ik},&
      \overset gT{}_{i.jk}=\overset gT{}_{k.ji}.
    \end{eqnarray*}

    That means the torsion tensor $\overset gT{}_{i.jk}$ may be expressed as

    \begin{equation*}
      \overset gT{}_{i.jk}=t_{ijk}+\tau_{ijk},
    \end{equation*}

    \noindent for a geometrical object $t_{ijk}$ anti-symmetric in any pair of
    indices $i,j,k$ and $\tau_{ijk}=\tau_{kji}$. For this reason, we get

    \begin{equation}
    \aligned
      \tau_{ijk}&=\frac12\big(\overset gT{}_{i.jk}+\overset gT{}_{k.ji}\big)=
      \frac12\big(\cancel{\partial_kg_{\underset\vee{ji}}}\,\bcancel{-
      \partial_ig_{\underset\vee{jk}}}\,\xcancel{+
      \partial_jg_{\underset\vee{ik}}}\,\bcancel{+
      \partial_ig_{\underset\vee{jk}}}\,\cancel{-
      \partial_kg_{\underset\vee{ji}}}\,\xcancel{+
      \partial_jg_{\underset\vee{ki}}}\big)=0.
    \endaligned
    \end{equation}

    In this way, we proved that
    the torsion tensor $\overset gT{}_{i.jk}$ is totally
    anti-symmetric.

    Assume that the covariant torsion tensor $\hat T$ with the components $\overset gT{}_{i.jk}$ is a function
    of some $\partial_rg_{\underline{pq}}$. Because
    $\overset gT{}_{i.jk}$ is anti-symmetric in any pair of the
    indices $i,j,k$ but the geometrical object
    $\partial_rg_{\underline{pq}}$ is symmetric in the indices $p$
    and $q$, these indices should be mute (dummy) in the functional
    correspondence between the objects
    $\partial_rg_{\underline{pq}}$ and $\overset gT{}_{i.jk}$.

    That means that it is satisfied the equation

    \begin{equation*}
      \overset gT{}_{i.jk}=\partial_jg_{\underline{(i)(k)}}q^{(i)(k)}_{ik}+\theta_{i.jk},
    \end{equation*}

    \noindent for some geometrical object
    $\theta_{i.jk}$ which is not a function of $\partial_jg_{\underline{ik}}$.
    Because the symmetric
    metric $\underline{\hat g}$ in this lemma
    is diagonal, the first summand in the last equation is not equal
    zero if and only if $i=k$. Because the component $\overset
    gT{}_{i.ji}=0$, i.e. it is the trivial function, its functional
    derivative by any function is equal $0$.
    \end{proof}

%  The full Einstein-Hilbert action is the integral
%
%  \begin{equation}
%    S=\int{d^4x\sqrt{|g|}\left(\frac1{2\kappa}\overset gR+\frac14
%    g^{\underline{\gamma\delta}}\overset gT{}^\beta_{\gamma\alpha}
%    \overset gT{}^\alpha_{\delta\beta}\right)}.
%    \label{eq:EH-effectbasic}
%  \end{equation}

  \subsection{Variations of the action}

  The variation of the action with the Lagrangian density (\ref{eq:lagrangian}) is

  \begin{align}
    &\delta S=\int{d^4x\left[\left(\frac1{2\kappa}\overset gR+\frac14
    g^{\underline{\gamma\delta}}\overset gT{}^\beta_{\gamma\alpha}
    \overset gT{}^\alpha_{\delta\beta}\right)\delta\sqrt{|g|}+
    \sqrt{|g|}\delta\left(\frac1{2\kappa}\overset gR+\frac14
    g^{\underline{\gamma\delta}}\overset gT{}^\beta_{\gamma\alpha}
    \overset gT{}^\alpha_{\delta\beta}\right)\right]}.\label{eq:deltas=02}
  \end{align}

  Because
  $\delta\sqrt{|g|}=-\frac12\sqrt{|g|}g_{\underline{\alpha\beta}}
  \delta g^{\underline{\alpha\beta}}$, $\int{d^4x\sqrt{|g|}g^{\underline{\alpha\beta}}
  \delta\overset g R_{\alpha\beta}}=0$ (from standard calculations)
  and $\delta\overset g R=\overset g R_{\alpha\beta}\delta g^{\underline{\alpha\beta}}$,
  the equation
  (\ref{eq:deltas=02}) transforms to

  \begin{equation}
  \aligned
    \delta S&=-\frac12\int{d^4x\left(\frac1{2\kappa}\overset gR+\frac14
    g^{\underline{\gamma\delta}}\overset gT{}^\zeta_{\gamma\epsilon}
    \overset gT{}^\epsilon_{\delta\zeta}\right)\sqrt{|g|}g_{\underline{\alpha\beta}}
    \delta g^{\underline{\alpha\beta}}}\\&+
    \int{d^4x\frac1{2\kappa}\sqrt{|g|}\overset gR_{\alpha\beta}
    \delta g^{\underline{\alpha\beta}}}+
    \frac14\int{d^4x\sqrt{|g|}\frac{\delta\left(
    g^{\underline{\epsilon\zeta}}
    \overset gT{}^\gamma_{\epsilon\delta}\overset gT{}^\delta_{\zeta\gamma}\right)}{\delta
    g^{\underline{\alpha\beta}}}\delta g^{\underline{\alpha\beta}}},
  \endaligned\label{eq:deltas=03}
  \end{equation}

  \noindent i.e.

    \begin{equation}
  \aligned
    \delta S&=-\frac12\int{d^4x\left(\frac1{2\kappa}\overset gR+\frac14
    g^{\underline{\gamma\delta}}\overset gT{}^\zeta_{\gamma\epsilon}
    \overset gT{}^\epsilon_{\delta\zeta}\right)\sqrt{|g|}g_{\underline{\alpha\beta}}
    \delta g^{\underline{\alpha\beta}}}\\&+
    \int{d^4x\frac1{2\kappa}\sqrt{|g|}\overset gR_{\alpha\beta}
    \delta g^{\underline{\alpha\beta}}}+
    \frac14\int{d^4x\sqrt{|g|}\frac{\delta\left(
    g^{\underline{\epsilon\zeta}}
    g^{\underline{\rho\gamma}}
    g^{\underline{\sigma\delta}}
    \overset gT{}_{\rho.\epsilon\delta}\overset gT{}_{\sigma.\zeta\gamma}\right)}{\delta
    g^{\underline{\alpha\beta}}}\delta g^{\underline{\alpha\beta}}}.
  \endaligned\tag{\ref{eq:deltas=03}'}\label{eq:deltas=03'}
  \end{equation}

    Based on the equations (\ref{eq:Gammaantisimgantisim},
    \ref{eq:Gammaantisimgantisim'}), we
  conclude that the anti-symmetric part
  $\Gamma_{i.\underset\vee{jk}}$ is anti-symmetric by indices
  $i$ and $j$ as well as $j$ and $k$, but it is symmetric by indices
  $i$ and $k$.

        Let us define the scalar object $F(g^{\underline{ij}})=g^{\underline{\epsilon\zeta}}
    g^{\underline{\rho\gamma}}
    g^{\underline{\sigma\delta}}
    \overset gT{}_{\rho.\epsilon\delta}\overset gT{}_{\sigma.\zeta\gamma}$.
    Because
  $\overset gT{}_{i.jk}=-\overset gT{}_{i.kj}=-\overset gT{}_{j.ik}=
  \overset gT{}_{j.ki}=-\overset gT{}_{k.ji}=\overset gT{}_{k.ij}$,
  $g_{\underline{\alpha\beta}}\delta g^{\underline{\alpha\beta}}=
  -g^{\underline{\alpha\beta}}\delta g_{\underline{\alpha\beta}},
  g^{\underline{\epsilon\zeta}}=g^{\underline{\zeta\epsilon}},\delta\overset g
  T{}_{i.jk}/\delta g_{\underline{pq}}=0$ and
  $g_{\underline{\epsilon\zeta}}=g_{\underline{\zeta\epsilon}}$ we
  obtain

  \begin{eqnarray}
    \aligned
    \frac{\delta\Big(F(g_{\underline{ij}})\Big)}{\delta
    g_{\underline{\alpha\beta}}}&=
    2g^{\underline{\rho\gamma}}g^{\underline{\sigma\delta}}
    \overset gT{}_{\rho.\alpha\delta}\overset gT{}_{\sigma.\beta\gamma}+
    2g^{\underline{\epsilon\zeta}}g^{\underline{\sigma\delta}}
    \overset gT{}_{\alpha.\epsilon\delta}\overset gT{}_{\beta.\sigma\zeta}+
    2g^{\underline{\epsilon\zeta}}
    g^{\underline{\rho\gamma}}\overset gT{}_{\alpha.\rho\epsilon}
    \overset gT{}_{\beta.\zeta\gamma}=6
    \overset gT{}^\gamma_{\alpha\delta}\overset gT{}^\delta_{\beta\gamma}.
    \endaligned\label{eq:I4}
  \end{eqnarray}

  Hence, the equation (\ref{eq:deltas=03'}) is equivalent to

  \begin{equation}
    \delta S=\int{d^4x\sqrt{|g|}\left(-\frac1{4\kappa}
    \overset gRg_{\underline{\alpha\beta}}
    -\frac18g^{\underline{\gamma\delta}}
    \overset gT{}^\zeta_{\gamma\epsilon}\overset gT{}^\epsilon_{\delta\zeta}g_{\underline{\alpha\beta}}
    +\frac1{2\kappa}
    \overset gR{}_{\alpha\beta}+\frac23\overset gT{}^\gamma_{\alpha\delta}
    \overset gT{}^\delta_{\beta\gamma}\right)\delta
    g^{\underline{\alpha\beta}}}.
    \tag{\ref{eq:deltas=03'}'}\label{eq:deltas=03''}
  \end{equation}

  Variation of the action is zero, $\delta S=0$, if and only if

  \begin{equation}
    \overset gR{}_{ij}-\frac12\overset gRg_{\underline{ij}}=
    \kappa\left(\frac14g^{\underline{\gamma\delta}}
    \overset gT{}^\zeta_{\gamma\epsilon}\overset gT{}^\epsilon_{\delta\zeta}g_{\underline{ij}}-
    \frac{4}3\overset gT{}^\gamma_{i\delta}\overset gT{}^\delta_{j\gamma}\right).
    \label{eq:EH-Effect-general1}
  \end{equation}

  In this way we obtained the set of generalized Einstein
  equations.

  Because there are two linearly independent scalar curvatures of
  the space $\mathbb{G}\overset g{\mathbb R}_4$, for example $\underset0{\overset g{\widetilde R}}$
  and $\underset 5{\overset g{\widetilde R}}=\overset gR$, the generalized Einstein
  equations (\ref{eq:EH-Effect-general1}) reduce to the Einstein's
  equations

  \begin{equation}
    \overset gR{}_{ij}-\frac12\overset gRg_{\underline{ij}}=0,
    \label{eq:EH-Effect-general2}
  \end{equation}

  \noindent in the case of the lagrangian $\mathcal L=\sqrt{|g|}\underset
  5{\overset g{\widetilde{R}}}\equiv\sqrt{|g|}\overset gR$ ($\mathcal L_M$).

  The equations (\ref{eq:EH-Effect-general1},
  \ref{eq:EH-Effect-general2}) are the linearly independent
  dynamical
   equations for
  metric field with torsion. Any other equation of motion for models with torsion may be
  expressed as the corresponding linear combination of the equations
  (\ref{eq:EH-Effect-general1},
  \ref{eq:EH-Effect-general2}).

  Components of the corresponding energy-momentum tensor for the cosmological fluid are written
  on the right side of (\ref{eq:EH-Effect-general1})

  \begin{equation}
    \overset gT{}_{ij}=\frac1 4g^{\underline{\gamma\delta}}
    \overset gT{}^\zeta_{\gamma\epsilon}\overset gT{}^\epsilon_{\delta\zeta}g_{\underline{ij}}-
    \frac{4}3\overset gT{}^\gamma_{i\delta}\overset gT{}^\delta_{j\gamma}.
    \label{eq:stressenergydfn}
  \end{equation}

  With respect to the expressions
  (\ref{eq:genGammaGamma11}---\ref{eq:genGammaGamma44}),
  we get

  \begin{align}
    &\aligned \overset gT{}^\alpha_{0\beta}
    \overset gT{}^\beta_{0\alpha}=-2g^{-1}s_0(t)\Big(&\big(n_3'(t)\big)^2
    s_3(t)+
    \big(n_4'(t)\big)^2s_2(t)+
    \big(n_5'(t)\big)^2s_1(t)\Big),\endaligned\label{eq:TT11}\\\displaybreak[0]
    &\overset gT{}^\alpha_{1\beta}\overset gT{}^\beta_{1\alpha}=
    -2g^{-1}s_1(t)\Big(\big(n_3'(t)\big)^2s_3(t)+
      \big(n_4'(t)\big)^2s_2(t)\Big),\label{eq:TT22}\\\displaybreak[0]
    &\overset gT{}^\alpha_{1\beta}\overset gT{}^\beta_{2\alpha}=-
    2g^{-1}s_1(t)s_2(t)n_4'(t)n_5'(t),\label{eq:TT23}\\\displaybreak[0]
    &\overset gT{}^\alpha_{1\beta}\overset gT{}^\beta_{3\alpha}=
    2g^{-1}s_1(t)s_3(t)n_3'(t)n_5'(t),\label{eq:TT24}\\\displaybreak[0]
    &\overset gT{}^\alpha_{2\beta}\overset gT{}^\beta_{1\alpha}=-
    2g^{-1}s_1(t)s_2(t)n_4'(t)n_5'(t),\label{eq:TT32}\\\displaybreak[0]
    &\overset gT{}^\alpha_{2\beta}\overset gT{}^\beta_{2\alpha}=-2g^{-1}s_2(t)
    \Big(\big(n_3'(t)\big)^2s_3(t)+
      \big(n_5'(t)\big)^2s_1(t)\Big),\label{eq:TT33}\\\displaybreak[0]
    &\overset gT{}^\alpha_{2\beta}\overset gT{}^\beta_{3\alpha}=-2g^{-1}s_2(t)s_3(t)n_3'(t)n_4'(t),\label{eq:TT34}\\\displaybreak[0]
    &\overset gT{}^\alpha_{3\beta}\overset gT{}^\beta_{1\alpha}=2g^{-1}s_1(t)s_3(t)n_3'(t)n_5'(t),\label{eq:TT42}\\\displaybreak[0]
    &\overset gT{}^\alpha_{3\beta}\overset gT{}^\beta_{2\alpha}=-2g^{-1}s_2(t)s_3(t)n_3'(t)n_4'(t),\label{eq:TT43}\\\displaybreak[0]
    &\overset gT{}^\alpha_{3\beta}\overset gT{}^\beta_{3\alpha}=-2g^{-1}s_3(t)\Big(\big(n_4'(t)\big)^2
    s_2(t)+
      \big(n_5'(t)\big)^2s_1(t)\Big),\label{eq:TT44}
  \end{align}

  \noindent but $\overset gT{}^\alpha_{0\beta}\overset
  gT{}^\beta_{i\alpha}=
  \overset gT{}^\alpha_{i\beta}\overset gT{}^\beta_{0\alpha}=0$  for
  $i\in\{1,2,3\}$.

  Because the matrix $\hat{\underline g}{}^{-1}$ is diagonal, the
  following equation holds

  \begin{equation}
    g^{\underline{\gamma\delta}}
    \overset gT{}^\zeta_{\gamma\epsilon}\overset gT{}^\epsilon_{\delta\zeta}=-
    6g^{-1}\Big(\big(n_3'(t)
    \big)^2s_3(t)+\big(n_4'(t)\big)^2s_2(t)+
    \big(n_5'(t)\big)^2s_1(t)\Big).\label{eq:gTT}
  \end{equation}

  \subsection{General formulae}

    The components of the energy-momentum tensor for a non-ideal (cosmological) fluid
  are \cite{madsen1}

  \begin{equation}
    \overset gT{}_{ij}=(\overset gp+\overset g\rho)u_iu_j+\overset gp
    g_{\underline{ij}}+\Delta_{ij},
    \label{eq:nonidealfluidTgeneral}
  \end{equation}

  \noindent for the components $u_i$ of a velocity $\hat u$, the pressure $\overset gp$
  and the density $\overset g\rho$,
   as well as the tensor $\hat\Delta$ whose components $\Delta_{ij}$ are symmetric in the indices $i$ and $j$.
   In the case of an ideal fluid, (\ref{eq:nonidealfluidTgeneral}) reduces to

  \begin{equation}
    \overset gT{}_{ij}=(\overset gp+\overset g\rho)u_iu_j+\overset gpg_{\underline{ij}}.
    \label{eq:idealfluidTgeneral}
  \end{equation}

  From the M. S. Madsen's article \cite{madsen1}, we may read that the energy-momentum tensor
  for non-ideal fluid is expressed more explicitly as

  \begin{equation}
    \overset gT{}_{ij}=\overset g\rho u_iu_j+\overset gq{}_iu_j+
    \overset gq{}_ju_i-\big(\overset gph_{ij}+\pi_{ij}\big),
    \label{eq:stressenergymadsen}
  \end{equation}

  \noindent for the velocity components $u_i$ such that $u_\alpha u^\alpha=1$, the energy
  density $\overset g\rho=
  \overset gT{}_{\alpha\beta}u^\alpha u^\beta$,
  the geometrical objects $h_{ij}=g_{\underline{ij}}-u_iu_j$, $
  \overset gq{}_i=\overset gT{}_{\alpha\beta}u^\alpha h^\beta_i$, $\Pi_{ij}=
  \overset gph_{ij}+\pi_{ij}=
  -\overset gT{}_{\alpha\beta}h^\alpha_ih^\beta_j$, and the pressure $
  \overset gp=\frac13\Pi^\alpha_\alpha$.

  From the expression of the objects $\Pi_{ij}$ and $h_{ij}$, we get

  \begin{equation*}
    \Pi_{ij}=-\overset gT{}_{\alpha\beta}\big(\delta^\alpha_i-u^\alpha u_i\big)\big(
    \delta^\beta_j-u^\beta u_j\big),
  \end{equation*}

  \noindent which leads to

  \begin{equation}
    \Pi_{ij}=-\overset gT{}_{ij}+
    \overset gT{}_{i\beta}u^\beta u_j+
    \overset gT{}_{\alpha j}u^\alpha u_i-
    \overset gT{}_{\alpha\beta}u^\alpha u^\beta u_iu_j.
    \label{eq:PIijmadsen}
  \end{equation}

  If we multiply the last equation with $g^{\underline{ij}}$, we will obtain the following results

  \begin{equation*}
    \Pi^\alpha_\alpha=-\overset gT{}^\alpha_\alpha+
    \overset gT{}_{\alpha\beta}u^\beta u^\alpha+
    \overset gT{}_{\alpha\beta}u^\alpha u^\beta-
    \overset gT{}_{\alpha\beta}u^\alpha u^\beta u^\gamma u_\gamma,\end{equation*}

    \noindent i.e.

    \begin{equation}
    \Pi^\alpha_\alpha=-\overset gT{}^\alpha_\alpha+\overset gT{}_{\alpha\beta}u^\alpha u^\beta.
    \label{eq:PIalphaalphamadsen}
  \end{equation}

  Hence, the pressure $\overset gp$ of the fluid is

  \begin{equation}
    \overset gp=-\frac13\overset gT{}_{\alpha\beta}\big(g^{\underline{\alpha\beta}}-u^\alpha u^\beta\big)=
    -\frac13\big(\overset gT{}_{\alpha}^{\alpha}-\overset g\rho\big).
    \label{eq:pressurefinalmadsen}
  \end{equation}

    In the comoving reference frame $u_i=\delta_{i0}$, the last equation reduces to

    \begin{equation}
      \overset gp=-\frac13\overset gT{}^\alpha_\alpha+\frac13\overset gT{}_{00}.
      \tag{\ref{eq:pressurefinalmadsen}'}\label{eq:pressurefinalmadsen'}
    \end{equation}

  The following lemma holds.

  \begin{lem}
    The pressure $\overset gp$, the density $\overset g\rho$
    and the components of the energy-momentum
    tensor for a non-ideal fluid satisfy the equation \emph{(\ref{eq:pressurefinalmadsen})}.
    In the comoving reference frame, the pressure
    reduces to \emph{(\ref{eq:pressurefinalmadsen'})}.

    The energy density $\overset g\rho$ is expressed as
    \begin{equation}
    \overset g\rho=\overset gT{}_{\alpha\beta}
    u^\alpha u^\beta.
    \label{eq:rhogeneral}
    \end{equation}

    \noindent In the comoving reference frame the energy density reduces to

    \begin{equation}
      \overset g\rho=\overset gT{}_{00}.
      \tag{\ref{eq:rhogeneral}'}\label{eq:rhogeneral'}
    \end{equation}

    Having in mind that the usual equation of state for a fluid is given by

    \begin{equation*}
      \overset gp=\overset g\omega\cdot \overset g\rho,
    \end{equation*}

    \noindent where $\overset g\omega$ is the state parameter, then $\overset g\omega$
    can be expressed as

    \begin{align}
      &\overset g\omega=
      -\frac13\overset gT{}^\alpha_\alpha\left(\overset gT{}_{\beta\gamma}u^\beta u^\gamma\right)^{-1}+\frac13.
      \label{eq:omegafinal}
    \end{align}

    In the comoving reference frame the last expression reduces to

    \begin{equation}
      \overset g\omega=-\frac13\overset gT{}^\alpha_\alpha\big(\overset gT{}_{00}\big)^{-1}+\frac13.
      \label{eq:omegacfinal}
    \end{equation}
 \qed
  \end{lem}

  With respect to the previous computations
  and the last Lemma, we obtain that the next theorem holds:

  \begin{thm}\label{thm:3.1}
  The components $\overset gT{}_{ij}$ of the energy-momentum tensor
  and the components
  $\overset gT{}^i_{jk}$ of the torsion tensor    satisfy the equation

   \begin{equation}
    \overset gT{}_{ij}\overset{(\ref{eq:deltas=03''})}=\frac14\overset gT{}^\zeta_{\gamma\epsilon}\overset gT{}^\epsilon_{\delta\zeta}
    \left(g^{\underline{\gamma\delta}}g_{\underline{ij}}-
    \frac{16}3\delta^\gamma_i\delta^\delta_j\right).
    \label{eq:energy-momentum-tensor-factored}
  \end{equation}

  \pagebreak

  The trace of the energy-momentum tensor  is

  \begin{equation}
    \overset gT{}^\alpha_\alpha=2g^{-1}
      \Big(\big(n_3'(t)\big)^2s_3(t)+
      \big(n_4'(t)\big)^2s_2(t)+
      \big(n_5'(t)\big)^2s_1(t)\Big).
      \label{eq:Talphaalphageneralfinalthm}
  \end{equation}

    The energy-density is

  \begin{equation}
    \aligned
    \overset g\rho&=-\frac32g^{-1}\left(\big(n_3'(t)\big)^2s_3(t)+
    \big(n_4'(t)\big)^2s_2(t)+\big(n_5'(t)\big)^2s_1(t)
    \right)-\frac43\overset gT{}^\gamma_{\alpha\delta}
    \overset gT{}^\delta_{\beta\gamma}u^\alpha u^\beta.
    \endaligned\label{eq:densitygeneralfinalthm}
  \end{equation}

  The pressure is

  \begin{equation}
  \overset gp=-\frac76g^{-1}\left(\big(n_3'(t)\big)^2s_3(t)+
    \big(n_4'(t)\big)^2s_2(t)+\big(n_5'(t)\big)^2s_1(t)
    \right)-\frac49\overset gT{}^\gamma_{\alpha\delta}
    \overset gT{}^\delta_{\beta\gamma}u^\alpha u^\beta.
    \label{eq:pressuregeneralfinalthm}
  \end{equation}

  The components of the $1$-form $\hat{\overset gq}$ are

  \begin{equation}
    \overset gq{}_i=\overset gT{}_{\alpha i}u^\alpha-u_i
    \overset gT{}_{\alpha\beta}u^\alpha u^\beta.
    \label{eq:qigeneralfinalthm}
  \end{equation}

  The state parameter $\overset g\omega$ is

  \begin{equation}
    \aligned
    \overset g\omega&=\frac13\left(
    g^{-1}\left(\big(n_3'(t)\big)^2s_3(t)+
    \big(n_4'(t)\big)^2s_2(t)+\big(n_5'(t)\big)^2s_1(t)
    \right)+\frac49\overset gT{}^\gamma_{\alpha\delta}
    \overset gT{}^\delta_{\beta\gamma}u^\alpha u^\beta
    \right)\\&\cdot
    \left(\frac32g^{-1}\left(\big(n_3'(t)\big)^2s_3(t)+
    \big(n_4'(t)\big)^2s_2(t)+\big(n_5'(t)\big)^2s_1(t)
    \right)+\frac43\overset gT{}^\gamma_{\alpha\delta}
    \overset gT{}^\delta_{\beta\gamma}u^\alpha u^\beta\right)^{-1}.
    \endaligned\label{eq:omegageneralfinalthm}
  \end{equation}

    The components  of the energy-momentum tensor,
    the torsion-tensor, and the metric
    tensor
    satisfy the equation

     \begin{equation}
       \aligned
       \overset gT{}_{ij}&=-\frac{3}{2}
       g^{-1}\Big(\big(n_3'(t)
    \big)^2s_3(t)+\big(n_4'(t)\big)^2s_2(t)+
    \big(n_5'(t)\big)^2s_1(t)\Big)g_{\underline{ij}}-\frac{4}3
    \overset gT{}^\alpha_{i\beta}\overset gT{}^\beta_{j\alpha},
       \endaligned\label{eq:energy-momentum-tensor-thm}
     \end{equation}

     \noindent for the above obtained $\overset gT{}^\alpha_{i\beta}\overset gT{}^\beta_{j\alpha}$.\hfill\qed
  \end{thm}

  \begin{cor}\label{cor:3.1}

  In the comoving reference frame, $u_i=\delta_{i0}$, the expressions
  for $\overset g\rho,\overset gp,\overset gq{}_i$ and $\overset g\omega$, reduce to

  \begin{align}
  &  \aligned
    \overset gp&=-\frac{1}{18}g^{-1}\big(21+8s_0(t)\big)
    \left(\big(n_3'(t)\big)^2s_3(t)+\big(n_4'(t)\big)^2s_2(t)+
    \big(n_5'(t)s_1(t)\big)^2\right),
    \endaligned\label{eq:pressurecgeneralfinalcor}\\
  &\aligned
    \overset g\rho&=\overset gT{}_{00}=-\frac16g^{-1}\big(9-16s_0(t)\big)
    \left(\big(n_3'(t)\big)^2s_3(t)+
    \big(n_4'(t)\big)^2s_2(t)+\big(n_5'(t)\big)^2s_1(t)
    \right),
  \endaligned\label{eq:rhocgeneralfinalcor}\\
&\aligned    \overset gq_i=\left\{\begin{array}{ll}
      -\frac16g^{-1}\big(9-16s_0(t)\big)
    \left(\big(n_3'(t)\big)^2s_3(t)+
    \big(n_4'(t)\big)^2s_2(t)+\big(n_5'(t)\big)^2s_1(t)
    \right),&i=0,\\
    0&{i>0},
    \end{array}\right.
  \endaligned\label{eq:qicgeneralfinal}\\
 & \aligned \overset g\omega=\frac{1}{3}\big(21+8s_0(t)\big)\big(9-16s_0(t)\big)^{-1}.
    \endaligned\label{eq:omegacgeneralfinalcor}
  \end{align}
     \qed
  \end{cor}

  \section{Some special cases}\label{section4physarticle1}

In this section, we examine Friedmann-like and Bianchi type-I-like models with torsion.

  \subsection{Friedmann spacetime with torsion}

      Let be $s_0(t)\equiv-1,s_1(t)=s_2(t)=s_3(t)=s(t),
    n_3(t)=n_4(t)=n_5(t)=n(t)$. In this case,  it is satisfied the
    following equations

    \begin{align}
      &\overset gT{}_{00}=9.5\big(n'(t)\big)^2
      \big(s(t)\big)^{-2}
      ,\label{eq:densitysn3}\\\displaybreak[0]
      &\overset gT{}_{11}=4\big(n'(t)\big)^2\big(s(t)\big)^{-1},\\&
      \overset gT{}_{12}=\overset gT{}_{21}=3\big(n'(t)\big)^2\big(s(t)\big)^{-1},
      \overset gT{}_{13}=\overset gT{}_{31}=-2\big(n'(t)\big)^2\big(s(t)\big)^{-1},
      \label{eq:T11sn3}\\\displaybreak[0]
      &\overset gT{}_{22}=4\big(n'(t)\big)^2\big(s(t)\big)^{-1},
      \overset gT{}_{23}=\overset gT{}_{32}=2\big(n'(t)\big)^2\big(s(t)\big)^{-1},
      \label{eq:T22sn3}\\\displaybreak[0]
      &\overset gT{}_{33}=4\big(n'(t)\big)^2\big(s(t)\big)^{-1}
      ,
      \label{eq:T33sn3}
    \end{align}

    \noindent and $\overset gT{}_{i0}=\overset gT{}_{0i}=0$ for $i\in\{1,2,3\}$.

    In the comoving reference frame, we get

%    \begin{align}
%      &\overset gp\overset{(\ref{eq:pressurecgeneralfinalcor})}=
%      \frac{39}{18}\big({n'(t)}\big)^2\big({s(t)}\big)^2,\label{eq:specc1p}\\
%      &\overset g\rho=\frac{25}2\big({n'(t)}\big)^2\big({s(t)}\big)^2,\label{eq:specc1rho}\\
%      &\overset g\omega=\frac{13}{75}\approx0.173.\label{eq:specc1omega}
%   \end{align}
%ovde
\begin{eqnarray}
  \overset gp\overset{(\ref{eq:pressurecgeneralfinalcor})}=
      \frac{39}{18}\big({n'(t)}\big)^2\big({s(t)}\big)^2,%\label{eq:specc1p}\\
      &\overset g\rho
      \overset{(\ref{eq:rhocgeneralfinalcor})}=\frac{25}2\big({n'(t)}\big)^2\big({s(t)}\big)^2,%\label{eq:specc1rho}\\
      &\overset g\omega
      \overset{(\ref{eq:omegacgeneralfinalcor})}=\frac{13}{75}\approx0.173.
\end{eqnarray}

    {}\bigskip

        Let be $s_0(t)\equiv-1,s_1(t)=s_2(t)=s_3(t)=s(t),
    n_3(t)=n_4(t)=n(t),n_5(t)=0$. In this case, we get:

    \begin{align}
      &\overset gT{}_{00}=\frac{17}3\big(n'(t)\big)^2\big(s(t)\big)^{-2}
      ,\label{eq:densitysn2}\\\displaybreak[0]
      &\overset gT{}_{11}=\frac{1}{3}\big(n'(t)\big)^2\big(s(t)\big)^{-1},
      \label{eq:T11sn2}\\\displaybreak[0]
      &\overset gT{}_{22}=\frac{1}{3}
      \big(n'(t)\big)^2\big(s(t)\big)^{-1},
      \overset gT{}_{23}=\overset gT{}_{32}=2\big(n'(t)\big)^2
      \big(s(t)\big)^{-1},
      \label{eq:T22sn2}\\\displaybreak[0]
      &\overset gT{}_{33}=-\frac{23}{12}\big(n'(t)\big)^2
      \big(s(t)\big)^{-1},
      \label{eq:T33sn2}
    \end{align}

    \noindent and $\overset gT{}_{ij}=\overset gT{}_{ji}=0$ in all other cases for
    $i,j\in\{0,1,2,3\}$.

    In the comoving reference  frame, we also obtain

%    \begin{align}
%      &\overset gp\overset{(\ref{eq:pressurecgeneralfinalcor})}=
%      \frac{13}9\big({n'(t)}\big)^2\big({s(t)}\big)^{-2},\label{eq:specc2p}\\
%      &\overset g\rho=\frac{25}3\big({n'(t)}\big)^2\big({s(t)}\big)^{-2},\label{eq:specc2rho}\\
%      &\overset g\omega=\frac{13}{75}\approx0.173.\label{eq:specc1omega}
%    \end{align}

\begin{eqnarray}
  \overset gp\overset{(\ref{eq:pressurecgeneralfinalcor})}=
      \frac{13}9\big({n'(t)}\big)^2\big({s(t)}\big)^{-2},
            &\overset g\rho
            \overset{(\ref{eq:rhocgeneralfinalcor})}=\frac{25}3\big({n'(t)}\big)^2\big({s(t)}\big)^{-2},
            &\overset g\omega
            \overset{(\ref{eq:omegacgeneralfinalcor})}=\frac{13}{75}\approx0.173.
\end{eqnarray}

  \subsection{Bianchi type-I spacetime  with torsion}

  The Bianchi type-I cosmological model is the main subject for our
  observations in this subsection. This model is characterized by
  the first square form

  \begin{equation}
  ds^2=-dt^2+s_1(t)\left.\!dx^1\right.^2{}+
  s_2(t)\left.\!dx^2\right.^2+s_3(t)\left.\!dx^3\right.^2.
  \end{equation}

  All above obtained results may be treated as generalizations of the Bianchi type-I spatetime model with torsion. In this section, we will generalize the Bianchi type-I spatetime model with a bidiagonal metrics.

    As we may see from the results above, the components $n_0(t),n_1(t),n_2(t)$ do not affect the action.
    For this reason, we will pay attention to the following metric

  \begin{equation}
    \hat b=\left[\begin{array}{cccc}
      -1&0&0&0\\
      0&s_1(t)&c(t)&0\\
      0&-c(t)&s_2(t)&0\\
      0&0&0&s_3(t)
    \end{array}\right].
    \label{eq:fmetrics}
  \end{equation}

  The symmetric and anti-symmetric part of $\hat b$ are

  \begin{eqnarray}
    \underline{\hat b}{}=
    \left[\begin{array}{cccc}
      -1&0&0&0\\
      0&s_1(t)&0&0\\
      0&0&s_2(t)&0\\
      0&0&0&s_3(t)
    \end{array}\right]&\mbox{and}&
    \underset\vee{\hat b}=\left[\begin{array}{cccc}
      0&0&0&0\\
      0&0&c(t)&0\\
      0&-c(t)&0&0\\
      0&0&0&0
    \end{array}\right].
    \label{eq:fmetricssimantisim}
  \end{eqnarray}

  The determinant of the matrix $\underline{\hat b}$ is
  $f=-s_1(t)s_2(t)s_3(t)$.

  We obtain the following values of energy-momentum tensor components in this case:

  \begin{align}
  &\overset bT{}_{00}=
  \dfrac{25}{6}\big(s_1(t)s_2(t)\big)^{-1}\big(c'(t)\big)^2,
  \label{eq:frho}\\
  &\overset bT{}_{11}=
  -\dfrac{7}{6}\big(s_2(t)\big)^{-1}\big(c'(t)\big)^2,
    \label{eq:hT11}\\
  &\overset bT{}_{22}=-\frac{7}{6}\big(s_1(t)\big)^{-1}\big(c'(t)\big)^2,
    \label{eq:hT22}\\
  &\overset bT{}_{33}=\frac{3}{2}\big(s_1(t)s_2(t)\big)^{-1}s_3(t)\big(c'(t)\big)^2,
    \label{eq:hT33}
  \end{align}

  \noindent and $\overset bT{}_{ij}=0$ in all other cases.

  In the comoving reference frame, we get the following values
%
%  \begin{align}
%      &\overset bp\overset{(\ref{eq:pressurecgeneralfinalcor})}=
%      \frac{13}{18}\Big(s_1(t)s_2(t)\Big)^{-1}\big(c'(t)\big)^2,\label{eq:speccbp}\\
%      &\overset b\rho=\frac{25}6\Big(s_1(t)s_2(t)\Big)^{-1}\big(c'(t)\big)^2,\label{eq:speccbrho}\\
%      &\overset b\omega=\frac{13}{75}.\label{eq:speccbomega}
%    \end{align}
%
\begin{eqnarray}
  \overset bp\overset{(\ref{eq:pressurecgeneralfinalcor})}=
      \frac{13}{18}\Big(s_1(t)s_2(t)\Big)^{-1}\big(c'(t)\big)^2,
      &\overset b\rho
      \overset{(\ref{eq:rhocgeneralfinalcor})}=\frac{25}6\Big(s_1(t)s_2(t)\Big)^{-1}\big(c'(t)\big)^2,
      &\overset b\omega
      \overset{(\ref{eq:omegacgeneralfinalcor})}=\frac{13}{75}.
\end{eqnarray}

\subsubsection{The Friedmann-Lemaitre-Robertson-Walker model}

In this part of the paper, we are aimed to consider the Friedmann-Lemaitre-Robertson-Walker cosmological model as a special case of the Bianchi type-I model.

  The metric (\ref{eq:fmetrics}) reduces to the metric for the FLRW
  model in the case of $s_1(t)=s_2(t)=s_3(t)=s(t)$. Because our
  previous results for pressure and density are
  obtained for general non-symmetric metric, we conclude that the results for
  FLRW model
  are the special cases of the results obtained for the Bianchi type-I
  model. In this case, i.e. in the case of the metric

  \begin{equation}
    \hat f=\left[\begin{array}{cccc}
      -1&0&0&0\\
      0&s(t)&c(t)&0\\
      0&-c(t)&s(t)&0\\
      0&0&0&s(t)
    \end{array}\right]
  \end{equation}

  \noindent we have
  \begin{eqnarray}
  \begin{array}{ccc}
    \overset
    fT{}_{00}=\dfrac{25}6\cdot\dfrac{\big(c'(t)\big)^2}{\big(s(t)\big)^2},&
    \overset fT{}_{11}=\overset
    fT{}_{22}=-\dfrac76\cdot\dfrac{\big(c'(t)\big)^2}{s(t)},&
    \overset fT{}_{33}=\dfrac32\cdot\dfrac{\big(c'(t)\big)^2}{s(t)},\\
    \multicolumn{3}{c}{\overset
    fp=\dfrac{13}{18}\cdot\dfrac{\big(c'(t)\big)^2}{\big(s(t)\big)^2}, \quad
    \overset
    f\rho=\dfrac{25}6\cdot \dfrac{\big(c'(t)\big)^2}{\big(s(t)\big)^2}, \quad
    \overset f\omega=\dfrac{13}{75}.}
  \end{array}
  \end{eqnarray}

    \section{Conclusion}

  After recalled the necessaries from differential geometry
  (Section \ref{section2physarticle1}),
  we geometrically studied and generalized the concept of the energy-momentum tensor.

  In the Section \ref{section3physarticle1} of this paper,
  we obtained the general formulae of the energy-momentum tensor, the pressure and the density with respect to the non-symmetric metrics and the torsion-tensor. After that, we analyzed the energy-momentum tensors obtained with respect to the special non-symmetric metrics.

  In the Section \ref{section4physarticle1}, we applied the general formulae to generalize the
   Bianchi type-I spacetime model. We studied the special case of the
   non-symmetric metrics with non-zero components placed on both diagonals
   of the non-symmetric metric matrices.

It is worth mentioning that although the state parameters for
Friedmann and Bianchi type-I spacetime model with torsion have the
same numerical value, it is obvious that those two cosmological
fluids have different dynamics, i.e. their pressures and densities
are different.

   Based on the results presented in this paper, in the future we will examine how
   many generalized Riemannian spaces in the Eisenhart's sense
   \big(see \cite{eisGRN1, eisGRN2}\big) does a lagrangian $\mathcal
   L_M\neq0$ generate. We will try to give the physical
   interpretations of these results. Moreover, we will
   use different concepts of the generalized Riemannian space
   \big(the Eisenhart's model used in this paper is one of them\big)
   to expand the Shapiro's model of the cosmology with torsion
   \big(presented in \cite{shapiro}\big).
      
\end{document}